\documentclass{amsart}
\usepackage{graphicx}

\usepackage[english]{babel}

\usepackage{amsmath, amsfonts, amssymb, graphicx, color, url, mathtools, bbold}

\usepackage[inline]{enumitem}

\usepackage{amsthm}

\usepackage{blkarray}

\usepackage[mathcal]{euscript} 

\usepackage[hang,flushmargin]{footmisc}

\usepackage[colorlinks, backref=page]{hyperref}

\usepackage{todonotes}

\newtheorem{theorem}{Theorem}[section]
\newtheorem{proposition}[theorem]{Proposition}
\newtheorem{corollary}[theorem]{Corollary}
\newtheorem{lemma}[theorem]{Lemma}

\theoremstyle{definition}
\newtheorem{definition}[theorem]{Definition}
\newtheorem{example}[theorem]{Example}

\newtheorem{remark}[theorem]{Remark}

\newcommand{\0}{\mathbf{0}}

\newcommand{\T}{\intercal} 

\newcommand{\kk}{\mathbb{k}}

\newcommand{\RR}{\mathbb{R}}

\newcommand{\ZZ}{\mathbb{Z}}

\newcommand{\cD}{\mathcal{D}}
\newcommand{\cF}{\mathcal{F}}
\newcommand{\cJ}{\mathcal{J}}

\newcommand{\cI}{\mathcal{I}}

\newcommand{\cP}{\mathcal{P}}
\newcommand{\cQ}{\mathcal{Q}}
\newcommand{\cS}{\mathcal{S}}
\newcommand{\cT}{\mathcal{T}}

\newcommand{\bfG}{\mathbf{G}}
\newcommand{\bfL}{\mathbf{L}}

\newcommand{\frC}{\mathfrak{C}}

\newcommand{\matroid}[1]{M(#1)}

\newcommand{\dual}[1]{{#1}^{\perp}} 
\newcommand{\dualM}{\dual{M}} 

\newcommand{\rank}[1]{r_{#1}} 
\newcommand{\rankdual}[1]{\rank{\dual{#1}}} 
\newcommand{\rankM}{\rank{M}} 
\newcommand{\rankdualM}{\rankdual{M}} 

\newcommand{\tutte}[1]{\cT_{#1}} 

\DeclareMathOperator{\rs}{rowspace}

\newcommand{\rankA}{\rank{\matroid{A}}}

\newcommand{\repn}[1]{V({#1})} 
\newcommand{\A}[1]{A({#1})} 

\newcommand{\nnZ}{\ZZ_{\geq 0}} 
\newcommand{\nnZd}{\nnZ^d} 

\newcommand{\nnR}{\RR_{\geq 0}} 
\newcommand{\nnRd}{\nnR^d} 

\newcommand{\nonparking}[1]{\mathcal{N}(#1)} 

\newcommand{\submod}[1]{f_{#1}} 

\newcommand{\PM}[1]{\frC({#1})} 
\newcommand{\PMA}{\PM{A}} 

\DeclareMathOperator{\im}{Im} 

\DeclareMathOperator{\Hilb}{Hilb} 

\newcommand{\polyyn}{\kk[y_1,\ldots,y_n]}

\newcommand{\polyx}{\kk\left[x_j\colon j\in {[d]}\right]} 

\newcommand{\polyy}{\kk\left[y_s\colon s\in S\right]} 

\newcommand{\SRI}[1]{\cS(#1)} 
\newcommand{\pIdeal}[1]{\cI(#1)} 
\newcommand{\pAlgebra}[1]{\cP(#1)} 
\newcommand{\sfalgebra}[1]{\cF_{#1}} 
\newcommand{\row}{r} 
\newcommand{\ringH}[1]{\Phi_{#1}}
\newcommand{\cIdeal}[1]{\cJ(#1)} 
\newcommand{\cAlgebra}[1]{\cD(#1)} 

\newcommand{\lPower}[1]{{\ell_{#1}}^{\rho_{#1}}} 
\newcommand{\lPowerH}{\lPower{H}} 

\newcommand{\gl}[2]{\bfG\bfL_{#1}(#2)}
\newcommand{\glk}[1]{\gl{#1}{\kk}}
\newcommand{\glnk}{\glk{n}}

\DeclareMathOperator{\Span}{span} 

\definecolor{darkblue}{rgb}{0,0,0.7} 
\newcommand{\darkblue}{\color{darkblue}} 
\newcommand{\defn}[1]{\emph{\darkblue #1}} 

\begin{document}

\title{On power ideals of transversal matroids and their ``parking functions''}

\author{Camilo Sarmiento}
\address{Universit\"at Leipzig, Mathematisches Institut, Leipzig, Germany}

\begin{abstract}
To a vector configuration one can associate a polynomial ideal generated by powers of linear forms, known as a power ideal, which exhibits many combinatorial features of the matroid underlying the configuration.

In this note we observe that certain power ideals associated to transversal matroids are, somewhat unexpectedly, monomial. Moreover, the (monomial) basis elements of the quotient ring defined by such a power ideal can be naturally identified with the lattice points of a remarkable convex polytope: a polymatroid, also known as generalized permutohedron. We dub the exponent vectors of these monomial basis elements ``parking functions'' of the corresponding transversal matroid.

We highlight the connection between our investigation and Stanley-Reisner theory, and relate our findings to Stanley's conjectured necessary condition on matroid $h$-vectors.
\end{abstract}

\maketitle


\section{Introduction}

Polynomial ideals generated by powers of linear forms, often called \defn{power ideals}, appear in a number of mathematical contexts. 
Notably, the dimensions of graded pieces of power ideals are the main object of study in investigations relating to Waring's problem for polynomial rings and, by work of Emsalem and Iarrobino, to ideals of fat points (see~\cite{Ger96}). 

This paper is concerned with a family of power ideals associated to a vector configuration. These were originally introduced in the context of multivariate approximation theory, mainly as a tool to study the space spanned by the local polynomial pieces of a box spline and their derivatives~\cite{dBDR91}. 
Such power ideals are known to strongly reflect combinatorial aspects of the underlying vector configuration (see e.g. Theorem~\ref{thm:tutteHilbert} below), and have generated renewed interest in recent years, owing to their rich geometry and combinatorics and to their relevance in subjects as varied as the cohomology of homogeneous manifolds and Cox rings (see~\cite{AP10, DP11, HR11} and the references therein.)
We delay a precise definition until Section~\ref{sec:pIdeals}.

In~\cite{PS04} Postnikov and Shapiro introduced and investigated a class of power ideals associated to graphs. 
Their definition can actually be seen to coincide with the one from multivariate approximation theory when a suitable vector configuration associated to a graph is taken (they were apparently unaware of such developments.)
Alongside they introduced a monomial ideal associated to a graph, and showed that both power and monomial ideals of a graph define graded quotient rings with the same Hilbert function, and vector space dimension equal to the number of spanning trees of the graph. 
Remarkably, the standard basis elements modulo the monomial ideal also form a basis for the quotient ring defined by the power ideal; their exponent vectors received the name ``$G$-parking functions'', as they specialize to the renowned parking functions. 
$G$-parking functions turn out to be intimately related to the chip-firing game on a graph, and have attracted much attention (see~\cite{BS13} and the references therein).

The motivation for this note is an attempt of the author to extend the methods of Postnikov and Shapiro beyond 
graphs. For this purpose, an obvious candidate to contemplate is a class of vector configurations associated to transversal matroids (see Section~\ref{sec:transmats} for definitions). Thus our subject matter is a family of power ideals associated to transversal matroids, with a focus on monomial bases for the quotient rings they define. 

It came as a surprise to the author that no monomial basis for such a power ideal needs to be constructed, in the first place: the power ideals we consider \emph{are} monomial (Theorem~\ref{thm:powerismono}), even though this is not at all evident from their generators, which comprise powers of linear forms of varying degrees and supports. 

A second surprise came when it was realized that the exponent vectors of the standard monomial basis for such a (monomial) power ideal can be readily identified with  the lattice points of a polymatroid (a.k.a. generalized permutahedron), a convex polytope defined by a submodular function (Corollary~\ref{cor:polymatroidmonomials}). By a (largely recreational) parking interpretation of such non-negative integer vectors, and in analogy with the $G$-parking functions of Postnikov and Shapiro, we have dubbed them ``$A$-parking functions'' associated to the set system $A$ defining the transversal matroid (see Remark~\ref{rem:Parking}). We emphasize however that a chip-firing-like interpretation of $A$-parking functions is missing. This stands in contrast with $G$-parking functions, which arise as \emph{superstable configurations} in the chip-firing terminology.

Key to our results are the classical Hall's marriage theorem, its generalization to independent transversals of a set system by Rado, and a generalization of the latter to polymatroids by McDiarmid. We refer the reader to the books of Lov\'asz and Plummer~\cite{LP86} and of Schrijver~\cite{Schr03} for a comprehensive account of these topics.

The power ideal associated to a vector configuration $V$ defines a graded ring whose Hilbert function is known to coincide with the $h$-vector of the (abstract) simplicial complex of subsets $T$ of~$V$ such that $\Span(V\setminus T)=\Span(V)$\footnote{That is, the independence complex of the matroid dual to the matroid of $V$ (cf. Remark~\ref{rem:conjStanley}).}. In this light, we believe that power ideals of vector configurations are most naturally regarded in the framework of Stanley-Reisner theory of matroids; we expand on this point of view in Section~\ref{sec:SRtheory}. In particular, our findings imply a proof of Stanley's conjecture for the class of $h$-vectors of cotransversal matroids, different (but cognate, after all) from an earlier one by Oh~\cite{Oh13}. This connection is spelled out in Remark~\ref{rem:conjStanley}.

The outline of the paper is as follows. Sections~\ref{sec:pIdeals} and~\ref{sec:transmats} collect some elementary notions and results concerning vector configurations and transversal matroids, respectively. Section~\ref{sec:transmats} includes a sample computation to illustrate our main result (Example~\ref{ex:pIdealTransMat}).  Section~\ref{sec:mainresult} presents our main results mentioned above, namely Theorem~\ref{thm:powerismono} and Corollary~\ref{cor:polymatroidmonomials}. Finally, Section~\ref{sec:SRtheory} is a brief excursion into Stanley-Reisner theory of matroids, intended to frame our investigation on power ideals.

\subsection*{Notation and conventions}

In this note we only consider finite sets and collections, so we will drop explicit mention of the hypothesis ``finite'' throughout. Given a positive integer $n$, we use the notation $[n]:=\{1,\ldots,n\}$. Given a set $S$, $|S|$ denotes the cardinality of $S$, and the notation $2^S$ stands for the power set of $S$, that is, the set of subsets of $S$. The set of nonnegative reals is denoted by $\nnR$ and the set of nonnegative integers by $\nnZ$. Given an element $q=(q_1,\ldots,q_d)\in\nnZ^d$, we write $x^q$ for the monomial in $\polyx$ with \defn{exponent vector} $q$, that is, $x^q:=\prod_{j\in[d]}x_j^{q_j}$.

\subsubsection*{Acknowledgements}
The author would like to thank Thomas Kahle for stimulating conversations, which eventually sparked the author's interest in matroids, Stanley-Reisner theory, power ideals and related objects. Computations with SageMath~\cite{sagemath}, Macaulay2~\cite{M2} and Polymake~\cite{polymake:2000} were invaluable for this work. The author would like to sincerely thank their developers and contributors.

\section{Vector configurations, their matroids and their power ideals}
\label{sec:pIdeals}

Let $\kk$ be a field. A \defn{vector configuration} over $\kk$ is a labeled collection $V=(v_s\colon s\in S)\subset\kk^d$ of (not necessarily distinct) vectors, for some $d\in\nnZ$. 
In the following, quotienting by $\Span V^\perp$ if necessary, we shall assume without loss of generality that the vectors in $V$ span $\kk^d$.

A vector configuration $V=(v_s\colon s\in S)\subset\kk^d$ provides the basic paradigm for a matroid on a ground set $S$. 
Namely, $V$ defines a matroid $M:=\matroid{V}$ whose structure is determined by its \defn{rank function} $\rankM\colon 2^S \to \nnZ$, defined by the rule:
\[
T \mapsto \dim_\kk\Span\{v_s\colon s\in T\},\quad \text{for }T\subseteq S
\]
Matroids arising from vector configurations in such a way are said to be \defn{representable (over~$\kk$)}.
The \emph{rank of the matroid} $M$ is given by $\rankM(S)$. Note that by our assumption on $V$, the rank of $M$ equals $d$. 
A subset $T\subset S$ is said to be \defn{independent} in $M$ if $\rankM(T)=|T|$, and a \defn{basis} of $M$ if, in addition,~$|T|=d$.

Matroids possess a beautiful duality theory that vastly generalizes the notion of orthogonal complement of vector subspaces in linear algebra. Given a rank-$d$ matroid $M$ on $S$, its \defn{dual matroid} is a rank-$(|S|-d)$ matroid $\dual{M}$ on $S$ whose rank function is given by $\rank{\dualM}(T)=\rank{M}(S\setminus T)-\rankM(S)+|T|$, for $T\subseteq S$ (cf.~\cite[Section 5.2]{Cra86}). 
Its independent sets are therefore the sets $T\subseteq S$ such that $\rankM(S\setminus T)=d$.

Representations of $\dualM=\dual{\matroid{V}}$ over $\kk$ can be characterized as follows. Given a vector configuration $V\subseteq \kk^d$ as above, write $\rs(V)$ for the $d$-dimensional subspace of $\kk^{|S|}$ spanned by the rows of the $d\times|S|$ matrix whose columns are the vectors in $V$. Let $W=(w_s\colon s\in S)\subseteq \kk^{|S|-d}$ be a vector configuration. Then $\matroid{W}=\dual{\matroid{V}}$ if and only if $\rs(W)=\rs(V)^\perp$.

\emph{Caveat:} Since this note is exclusively concerned with representable matroids, we will often commit the following abuse of notation. In referring to the groundset of a matroid $\matroid{V}$ (or to subsets or elements thereof), we will interchangeably mean the label set $S$ or the vector collection $V$ (or subsets or elements thereof).

Recall that a \defn{flat of $V$} is a subset of $V$ of the form\footnote{By the preceding caveat, $V\cap L$ may thus refer to a subset of $V$ or to the corresponding index subset of $S$.} $V\cap L$, where $L$ is a linear hyperplane in $\kk^d$. 
If a flat $H=V\cap L$ is maximal with respect to inclusion, it is called a \defn{hyperplane of $V$}. 
In that case, we let \defn{$\ell_H$} denote any linear form defining $L$ and write \defn{$\rho_H:=|V\setminus H|$}.

\begin{definition}
The \defn{power ideal of $V$} is the following the ideal of $\polyx$: 
\begin{equation*}
\pIdeal{V}:= \left(\lPowerH\colon H\text{ hyperplane of }V\right).
\end{equation*}
The \defn{power algebra of $V$} is the quotient $\pAlgebra{V}:=\polyx/\pIdeal{V}$.
\end{definition}

Clearly, $\pAlgebra{V}$ is a graded $\kk$-algebra, in the sense that it admits a decomposition $\pAlgebra{V}=\bigoplus_{k\in\nnZ} \pAlgebra{V}_k$, with $\pAlgebra{V}_0\cong \kk$ and $\pAlgebra{V}_k\pAlgebra{V}_l\subset \pAlgebra{V}_{k+l}$ for $k,l\in\nnZ$. We refer to $\pAlgebra{V}_k$ as the \defn{$k$-th graded component} of $\pAlgebra{V}$, and define the \defn{Hilbert series} of $\pAlgebra{V}$ as the formal power series $\Hilb(\pAlgebra{V};z):=\sum_{k\in\nnZ }\dim_\kk \pAlgebra{V}_k$.

The \defn{Tutte polynomial} of a matroid $M$ can be defined via the rank function of $M$ as follows
\[
\tutte{M}(x,y):=\sum_{T\subseteq S}(x-1)^{\rankM(S)-\rankM(T)}(y-1)^{|T|-\rankM(T)}.
\]
Many enumerative invariants of a matroid $M$ arise as specializations of $\tutte{M}(x,y)$. It occupies a special position in the present context because of the following theorem, which illustrates the combinatorial nature of power ideals.

\begin{theorem}[\cite{AP10,dBDR91,HR11}]
\label{thm:tutteHilbert}
$\displaystyle\Hilb(\pAlgebra{V};z):=z^{|S|-d}\tutte{\matroid{V}}\left(1,{z}^{-1}\right).$
\end{theorem}

Following Ardila~\cite[Chapter 4]{Ard03} and Postnikov and Shapiro~\cite[Section 9]{PS04}, the power ideal of $V$ can be presented as the ideal of relations of certain ``squarefree algebra''. 
Concretely, denote by $\sfalgebra{V}$ the quotient of $\polyy$ by the relations:
\begin{align*}
y_s^2\quad &\text{for }s\in[n],\\
\prod_{s\in T}y_s \quad &\text{for }T\subseteq S\text{ cocircuit of }\matroid{V},
\end{align*}
and consider the $\kk$-algebra homomorphism $\ringH{V}\colon \polyx\to\sfalgebra{V}$ defined by:
\[
\ringH{V}\colon x_j \mapsto \row_j:=\sum_{s\in S}(v_s)_jy_s,
\]
where $(v_s)_j$ denotes the $j$-th coordinate of vector $v_s\in V$. 

\begin{lemma}[{\cite[Corollary~10.5]{PS04}}]
\label{lem:powerkernel}
$\displaystyle
\ker \ringH{V}=\pIdeal{V}.
$
\end{lemma}

The proof of Lemma~\ref{lem:powerkernel} is not directly relevant to our main result. However, we include one in Section~\ref{sec:SRtheory} to highlight the relation between power ideals and Stanley-Reisner theory.

\section{Transversal matroids and their vector configurations}
\label{sec:transmats}

Let $A$ be a \defn{set system} on a ground set $S$, that is, a labeled collection of subsets of a set~$S$. 
We use the notation $A=(\A{j}\colon j \in [d])$, where $d$ is a positive integer and $\A{j}\subseteq S$ for $j\in [d]:=\{1,\ldots, d\}$. 

A \defn{partial transversal} of $A$ is a subset $T\subseteq S$ whose elements belong to distinct members of $A$, that is, such that $T=\{s_j\colon j\in J\}$ for some $J\subseteq [d]$, where $s_j\in \A{j}$ for each $j\in J$. 
The partial transversals of $A$ constitute the independent sets of a matroid $\matroid{A}$ on the ground set $S$~\cite{Bru87}. 
Matroids arising from set systems in such a way are known as \defn{transversal matroids}. 
By removing subsets in $A$ if necessary, we may assume in the following that the rank of $\matroid{A}$ equals $d$; that this entails no loss of generality follows from~\cite[Lemma~5.1.1]{Bru87}.

The following well-known construction shows that transversal matroids are representable over $\RR$. For every $j\in[d]$ and $s\in S$ define scalars $v_{j,s}$ such that $v_{j,s}=0$ if and only if $s\notin \A{j}$, and the nonzero $v_{j,s}$'s are algebraically independent transcendentals over $\RR$. For every $s\in S$ define the vector $v_s:=(v_{j,s}\colon j\in [d])^\T\in\RR^d$. 
The resulting vector configuration $\repn{A}:=(v_s\colon s\in S)$ represents $\matroid{A}$ over $\RR$, that is, $\matroid{\repn{A}}=\matroid{A}$~\cite[Theorem 5.4.7]{Bru87}\footnote{In fact, every transversal matroid can be represented over a sufficiently large field. Concretely, if the order of the field is at least $\smash{|S|+\binom{|S|}{d-1}}$, then a representation of $\matroid{A}$ can be constructed by taking the nonzero $v_{j,s}$'s from a Zariski open dense subset of a suitable affine space (see~\cite{Atk72}).}.

Underlying this representation of transversal matroids is the following existence statement of partial transversals in the case when $|S|=d$. For future reference, we have supplemented it with the celebrated \defn{Hall's marriage theorem}, which asserts the equivalence of statements~\ref{it:transv} and~\ref{it:hall} below.

\begin{theorem}[{\cite[Theorem 8.2.1]{LP86}}]
\label{thm:hallsthm}
Let $A=(\A{j}\colon j \in [d])$ be a set system with $|S|=d$. Denote by $\det(\repn{A})$ the determinant of the $d\times d$ matrix whose columns are given by the vectors in $\repn{A}$. The following statements are equivalent:
\begin{enumerate}[label=\thelemma(\alph*)]
\item \label{it:transv} $A$ has a partial transversal of size $d$.
\item \label{it:nonzerodet} $\det(\repn{A})\neq 0$.
\item \label{it:hall} $|\A{J}|\geq |J|$ for every $J\subseteq [d]$.
\end{enumerate}
\end{theorem}

\begin{example}
\label{ex:pIdealTransMat}
Consider the set system $A=(\{1,2,7,9,10\},\{2,6,7,8\},\{2,3,4,5,6\})$. 
To construct the vector configuration $V:=\repn{A}$ representing $M:=\matroid{A}$, we may take the nonzero coordinates of vector $v_i\in V$ as $v_{j,i}=i^j$, for $i\in[10]$. Thus $V$ consists of the columns of the following matrix:
\[
\begin{pmatrix}1&
      {2}&
      0&
      0&
      0&
      0&
      {7}&
      0&
      {9}&
      {10}\\
      0&
      {4}&
      0&
      0&
      0&
      {36}&
      {49}&
      {64}&
      0&
      0\\
      0&
      {8}&
      {27}&
      {64}&
      {125}&
      {216}&
      0&
      0&
      0&
      0\\
      \end{pmatrix}
\]
The power ideal $\pIdeal{V}\subset\RR[x_1,x_2,x_3]$ is generated by the following powers
\begin{gather*}
({-2 {x}_{2}+{x}_{3}})^{6}, ({{x}_{2}})^{4},  ({-6 {x}_{2}+{x}_{3}})^{6},  ({{x}_{3}})^{5}, 
 ({-2 {x}_{1}+{x}_{2}})^{6}, ({-8 {x}_{1}+6 {x}_{2}-{x}_{3}})^{8},\\ ({-28 {x}_{1}+4 {x}_{2}+5
      {x}_{3}})^{8}, ({-4 {x}_{1}+{x}_{3}})^{8}, ({{x}_{1}})^{5}, 
      ({-7 {x}_{1}+{x}_{2}})^{6}, ({-42
      {x}_{1}+6 {x}_{2}-{x}_{3}})^{8}. 
\end{gather*}
The power algebra $\pAlgebra{V}=\kk[x_1,x_2,x_3]/\pIdeal{V}$ has Hilbert series given by:
\[
\Hilb(\pAlgebra{V};z)=1+3 z+6 z^{2}+10 z^{3}+14
     z^{4}+16 z^{5}+12 z^{6}+8 z^{7},
\]
which, according to Theorem~\ref{thm:tutteHilbert}, is a specialization of the Tutte polynomial of $M$:
\[
\begin{aligned}
\tutte{M}(x,y)=y^{7}& + 3 y^{6} + x y^{4} + 6 y^{5} + 2 x^{2} y^{2} + 4 x y^{3} + 9 y^{4} + x^{3} + 2 x^{2} y\\& + 7 x y^{2} + 10 y^{3} + 3 x^{2} + 6 x y + 7 y^{2} + 4 x + 4 y.
\end{aligned}
\]

A computation with Macaulay2 shows that $\pIdeal{V}$ is in fact monomial, since it admits a Gr\"obner basis consisting of the following monomials:
\begin{gather*}
{x}_{2}^{4}, {x}_{3}^{5}, {x}_{1}^{5}, {x}_{2}^{2} {x}_{3}^{4}, {x}_{2}^{3}
     {x}_{3}^{3}, {x}_{1}^{3} {x}_{2}^{3}, {x}_{1}^{4} {x}_{2}^{2}, {x}_{1}^{3}
     {x}_{2} {x}_{3}^{4}, {x}_{1}^{4} {x}_{3}^{4}, {x}_{1}^{3} {x}_{2}^{2}
     {x}_{3}^{3}, {x}_{1}^{4} {x}_{2} {x}_{3}^{3}
\end{gather*}

The following $70$ monomials constitute the standard basis for $\pAlgebra{V}$:
{\footnotesize
\begin{gather*}
1, {x}_{1}, {x}_{1}^{2}, {x}_{1}^{3}, {x}_{1}^{4}, {x}_{1}^{4} {x}_{2},
      {x}_{1}^{4} {x}_{2} {x}_{3}, {x}_{1}^{4} {x}_{2} {x}_{3}^{2}, {x}_{1}^{4}
      {x}_{3}, {x}_{1}^{4} {x}_{3}^{2}, {x}_{1}^{4} {x}_{3}^{3}, {x}_{1}^{3}
      {x}_{2}, {x}_{1}^{3} {x}_{2}^{2}, {x}_{1}^{3} {x}_{2}^{2} {x}_{3},
      {x}_{1}^{3} {x}_{2}^{2} {x}_{3}^{2}, {x}_{1}^{3} {x}_{2} {x}_{3},\\
      {x}_{1}^{3} {x}_{2} {x}_{3}^{2}, {x}_{1}^{3} {x}_{2} {x}_{3}^{3},
      {x}_{1}^{3} {x}_{3}, {x}_{1}^{3} {x}_{3}^{2}, {x}_{1}^{3} {x}_{3}^{3},
      {x}_{1}^{3} {x}_{3}^{4}, {x}_{1}^{2} {x}_{2}, {x}_{1}^{2} {x}_{2}^{2},
      {x}_{1}^{2} {x}_{2}^{3}, {x}_{1}^{2} {x}_{2}^{3} {x}_{3}, 
      {x}_{1}^{2}
      {x}_{2}^{3} {x}_{3}^{2}, {x}_{1}^{2} {x}_{2}^{2} {x}_{3}, {x}_{1}^{2}
      {x}_{2}^{2} {x}_{3}^{2},\\ {x}_{1}^{2} {x}_{2}^{2} {x}_{3}^{3}, {x}_{1}^{2}
      {x}_{2} {x}_{3}, {x}_{1}^{2} {x}_{2} {x}_{3}^{2}, {x}_{1}^{2} {x}_{2}
      {x}_{3}^{3}, {x}_{1}^{2} {x}_{2} {x}_{3}^{4}, {x}_{1}^{2} {x}_{3},
      {x}_{1}^{2} {x}_{3}^{2}, {x}_{1}^{2} {x}_{3}^{3},
       {x}_{1}^{2} {x}_{3}^{4},
      {x}_{1} {x}_{2}, {x}_{1} {x}_{2}^{2}, {x}_{1} {x}_{2}^{3}, {x}_{1}
      {x}_{2}^{3} {x}_{3},\\ {x}_{1} {x}_{2}^{3} {x}_{3}^{2}, {x}_{1} {x}_{2}^{2}
      {x}_{3}, {x}_{1} {x}_{2}^{2} {x}_{3}^{2}, {x}_{1} {x}_{2}^{2} {x}_{3}^{3},
      {x}_{1} {x}_{2} {x}_{3}, {x}_{1} {x}_{2} {x}_{3}^{2}, {x}_{1} {x}_{2}
      {x}_{3}^{3}, {x}_{1} {x}_{2} {x}_{3}^{4}, {x}_{1} {x}_{3}, {x}_{1}
      {x}_{3}^{2}, {x}_{1} {x}_{3}^{3}, {x}_{1} {x}_{3}^{4},\\ {x}_{2}, {x}_{2}^{2},
      {x}_{2}^{3}, {x}_{2}^{3} {x}_{3}, {x}_{2}^{3} {x}_{3}^{2}, {x}_{2}^{2}
      {x}_{3}, {x}_{2}^{2} {x}_{3}^{2}, {x}_{2}^{2} {x}_{3}^{3}, {x}_{2} {x}_{3},
      {x}_{2} {x}_{3}^{2}, {x}_{2} {x}_{3}^{3}, {x}_{2} {x}_{3}^{4}, {x}_{3},
      {x}_{3}^{2}, {x}_{3}^{3}, {x}_{3}^{4}
\end{gather*}
}
Their exponent vectors are the lattice points of the  following polytope:
\[
\begin{aligned}
\PMA:=\left\{(q_1,q_2,q_3)^\T\in\nnR^3 \right.\colon & q_1\leq 4,\ q_2\leq 3,\ q_3\leq 4,\\ & q_1+q_2\leq5,\ q_1+q_3\leq 7,\ q_2+q_3\leq 5,\\ &\left. q_1+q_2+q_3\leq 7\right\},
\end{aligned}
\]
and are depicted\footnote{A PDF file with a 3d model of this polytope is available as an ancillary file for the arXiv version.} in Figure~\ref{fig:Polymatroid}.
\begin{figure}[htbp]
\begin{minipage}[c]{0.4\textwidth}
\centering
\includegraphics[width=\textwidth]{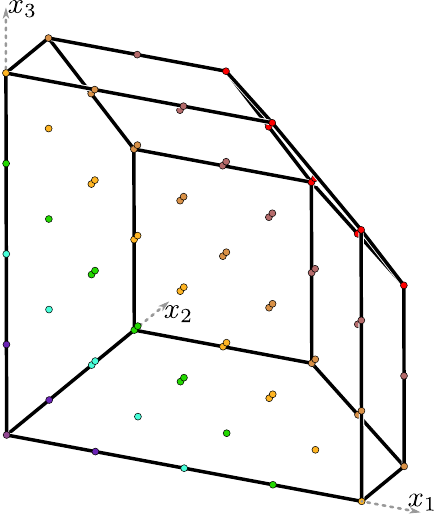}
\end{minipage}%
\begin{minipage}[c]{0.6\textwidth}
\caption{Polytope whose lattice points correspond to the exponent vectors of the standard monomial basis modulo the power ideal $\pIdeal{V}$ from Example~\ref{ex:pIdealTransMat}. Exponents of different degrees are shown in different colors to aid in visualization.}
\label{fig:Polymatroid}
\end{minipage}
\end{figure}
\end{example}

\section{Main result}
\label{sec:mainresult}

To present our main result, we fix a rank-$d$ set system $A=(\A{j}\colon j \in [d])$ on $S$, along with a representation $\repn{A}:=(v_s\colon s\in S)\subset\kk^d$, as constructed in Section~\ref{sec:transmats}. 
Given $J\subseteq[d]$ we write $A(J):=\bigcup_{j\in J}\A{j}$, and we denote by $\dual{M}$ the rank-$(|S|-d)$ matroid on $S$ dual to $\matroid{A}$, whose rank function we write $\rankdualM$ . 

\begin{theorem}
\label{thm:powerismono}
The power ideal $\pIdeal{\repn{A}}\subset\polyx$ is a monomial ideal.
\end{theorem}

Our proof of Theorem~\ref{thm:powerismono} proceeds directly, by identifying the monomial ideal $\nonparking{A}\subset\polyx$ that $\pIdeal{\repn{A}}$ is equal to. To this end, we introduce a set function defined on the subsets of $[d]$ as follows:
\begin{align*}
\submod{A}\colon 2^{[d]}&\to\nnZ \\
J&\mapsto \rankdualM(A(J)),\quad\text{ for }J\subseteq [d].
\end{align*}
It is not difficult to see that $\submod{A}(I\cap J)+\submod{A}(I\cup J)\leq \submod{A}(I)+\submod{A}(J)$ holds for every $I,J\subseteq [d]$, so that $\submod{A}$ is a \emph{submodular function} (see e.g.~\cite[Section 44.1a]{Schr03}). Like every submodular function, $\submod{A}$ defines a convex polytope known as a polymatroid. 
\begin{definition}
The \defn{parking polymatroid} of $A$ is the polymatroid $\PMA$ defined by $\submod{A}$, that is, the convex polytope in $\RR^d$ defined as follows:
\[
\PMA:=\left\{q\in\nnRd \colon \sum_{j\in J} q_j \leq \submod{A}(J)\text{ for every }J\subseteq [d]\right\}.
\]
\end{definition}

\begin{definition}
The \defn{nonparking ideal} of $A$ is the ideal $\nonparking{A}$ of $\polyx$ defined by the monomials $\{x^q\colon q\in \nnZd,\ q\notin \PMA\}$.
\end{definition}

\begin{remark}
\label{rem:Parking}
It is difficult to avoid the following ``parking'' interpretation of the lattice points of $\PMA$. 
A parking lot offers a set $S$ of labeled parking spots for cars of $d$ different brands, subject to the peculiar rule that parking spot $s$ may only be occupied by cars of the brands $J_s\subseteq [d]$, for $s\in S$. 
A number of cars totalling $q_i+1>0$ cars of brand $i$, for $i\in[d]$, arrive to park in this parking lot. 
We say that the tuple $q=(q_1,\ldots,q_d)\in\nnZd$ is an \defn{``$A$-parking function''} if  all the cars manage to park while observing the parking lot's rule.

To relate this to the polymatroid $\PMA$, let $\A{j}\subseteq S$ denote the subset of parking spots that cars of brand $j\in[d]$ may park in, and $A=(\A{j}\colon j \in [d])$ be the resulting set system\footnote{We assume that $A$ has rank $d$; that is, we neglect car brands which may not park at all.}. 
Notice that if $q\in\nnZd$ is an ``$A$-parking function'', then a collection of cars comprising $q_i\geq 0$ cars of brand $i$, for $i\in[d]$, can park in a subset $T\subset S$ of the parking spots, in such a way that at least one car of each brand can still find a parking spot among the remaining ones $S\setminus T$. 
In other words, $S\setminus T$ contains a partial transversal of size $d$ or, equivalently, $T$ is an independent set in the dual matroid $\dualM=\dual{\matroid{A}}$. 

Thus $q\in\nnZd$ is an ``$A$-parking function'' if and only if $A$ has a $q$-transversal that is independent in $\dualM$, where a \defn{$q$-transversal of $A$} is defined as a subset $T=T_1\cup\ldots\cup T_d\subseteq S$, such that $T_j\subseteq \A{j}$,  $|T_j|=q_j$ and $T_j\cap T_{j'}=\emptyset$ for $j,j' \in [d]$ distinct. 
By Rado's theorem for polymatroids (see \cite[Section 44.6g]{Schr03}), $A$ has a $q$-transversal that is independent in $\dualM$ if and only if $q\in\PMA$, that is, \mbox{if and only~if}
  \[
  \sum_{j\in J}q_j\leq \rankdualM\left( \A{J}\right)=:\submod{A}(J), \quad \text{holds for all }J\subseteq[d].
  \]

This parking analogy, together with the work of Postnikov and Shapiro on power ideals and parking functions associated to graphs~\cite{PS04}, motivated the chosen names for $\PMA$ and $\nonparking{A}$. 
We should point out, however, that the term ``$A$-parking function'' is chiefly understood as a nickname, because a chip-firing-like interpretation for it is currently unavailable.
It is a prominent problem in combinatorics to find variations and higher dimensional analogs of the chip-firing game on graphs.
\end{remark}

\begin{proposition}
$\pIdeal{\repn{A}}\subseteq \nonparking{A}$.
\end{proposition}

\begin{proof}

Let $H=\repn{A}\cap L$ be a hyperplane of $\repn{A}$ defined (up to scalar multiple) by the linear form $\ell_H\in\polyx$. 
Let $J\subseteq[d]$ be the subset of indices of the nonvanishing coefficients of $\ell_H$ and $q\in\nnZd$ be the exponent vector of a monomial in ${\ell_H}^{\rho_H}$. 
We claim that $\sum_{j\in J}q_j>\submod{A}(J)$, so that $x^q\in\nonparking{A}$.

Assume without loss of generality that $L=\Span\{v_{i_1},\ldots,v_{i_{d-1}}\}$ for some vectors $v_{i_1},\ldots,v_{i_{d-1}}\in\repn{A}$. Then $\ell_H$ can be written as the determinant of the matrix with columns given by the vectors $v_{i_1},\ldots,v_{i_{d-1}}$ and the vector $(x_1,\ldots,x_d)^\T$:
\[
\ell_H=\det\begin{pmatrix}
v_{i_1} &  \cdots & v_{i_{d-1}} & x_1 \\
\kern.6em\vline  &  & \kern.6em\vline & \vdots  \\
\kern.6em\vline  & &\kern.6em\vline  & x_d
\end{pmatrix}\in\polyx,
\]
and by Theorem~\ref{thm:hallsthm} it follows that its $j$-th coefficient is nonzero if and only if the set system\footnote{Recall that by our convention in Section~\ref{sec:pIdeals}, the notation $H$ interchangeably stands for the hyperplane $H$ and the subset of $S$ comprising the indices of vectors in $H$.} $\left(\A{j'}\cap H\colon j'\in [d]\setminus\{j\}\right)$ has a partial transversal of size $d-1$. 

By Hall's marriage theorem (cf. Theorem~\ref{thm:hallsthm}), this observation implies that $|J'|\leq|\A{J'}\cap H|$ whenever $J'\not\supseteq J$. Similarly, $|J|>|\A{J}\cap H|$ necessarily holds, because otherwise there would be some $j\in [d]\setminus J$ such that the set system $\left(\A{j'}\cap H\colon j'\in [d]\setminus\{j\}\right)$ has a partial transversal of size $d-1$, which contradicts the characterization of the vanishing coefficients of $\ell_H$. 
Now, since the $j$-th coefficient of $\ell_H$ vanishes whenever $j\notin J$, clearly $\ell_H(v_s)=0$ whenever $s\notin \A{J}$, and hence $S\setminus \A{J}\subseteq H$. Thus we find $|\A{J}|=|\A{J}\cap H|+|S\setminus H|=|\A{J}\cap H|+\rho_H$, which implies the inequality
\begin{equation}
\label{eq:hyperplane}
|\A{J}|<|J|+\sum_{j\in [d]}q_j=|J|+\sum_{j\in J}q_j
\end{equation}

On the other hand, we know that the rank function of a dual matroid can be written in terms of the rank function of the primal as follows (cf. Section~\ref{sec:pIdeals}):
\[
\rankdualM(T)=\rankA(S\setminus T)-\rankA(S)+|T|\text{ for }T\subseteq S
\]
Also, we know that the rank function of the transversal matroid defined by $A$ is given by (cf.~\cite[Proposition~4.2.3]{Bru87}):
\[
\rankA(T)=\min \{|\A{J'}\cap T|+d-|J'|\colon J'\subseteq [d]\} \text{ for }T\subseteq S.
\]
Combining the first equality evaluated at $\A{J}$ with the second one evaluated at $S\setminus \A{J}$, we obtain the following inequality:
\begin{equation}
\label{eq:dualrank}
\rankdualM(\A{J})\leq |\A{J}|-|J|.
\end{equation}
Equations~\eqref{eq:hyperplane} and~\eqref{eq:dualrank} then yield our claim that $\sum_{j\in J}q_j>\rankdualM(\A{J})$.
\end{proof}

\begin{proposition}
$\nonparking{A}\subseteq \ker \ringH{\repn{A}}$.
\end{proposition}

\begin{proof}
Let $q\in\nnZd$ be such that $x^q\in \nonparking{A}$. The squarefree monomials in the expansion of the image $\ringH{A}(x^q)\in \sfalgebra{\repn{A}}$ can be seen as $q$-transversals of $A$, as defined in Remark~\ref{rem:Parking}. Since $q\notin \PMA$, no such $q$-transversal is independent in $\dualM$ or, equivalently, every such $q$-transversal is divisible by $\prod_{s\in T}y_s$ for some cocircuit $T$ of $\matroid{A}$. Thus $\ringH{A}(x^q)=0$.
\end{proof}

\begin{corollary}
\label{cor:polymatroidmonomials}
$\{x^q\colon q\in\PMA\cap \nnZd\}$ forms a basis for $\pAlgebra{\repn{A}}$. 
\end{corollary}

\section{Connection with Stanley-Reisner theory}
\label{sec:SRtheory}

As in~\cite[Chapter 4]{Ard03} and~\cite[Section 9]{PS04}, our proof of Lemma~\ref{lem:powerkernel} draws upon a vector space of polynomials associated to $V$ whose dimension can be calculated easily. To introduce it, let $v_s(x):=\sum_{j\in [d]} (v_s)_j x_j$ denote the linear form in $\polyx$ with coefficients given by the coordinates of the vector $v_s\in V$.

\begin{definition}
The \defn{cocircuit ideal of $V$} is following the ideal of $\polyx$:
\[
\cIdeal{V}=\left(\prod_{s\in T} v_s(x)\colon T\subseteq S\text{ cocircuit of }\matroid{V}\right)
\]
The \defn{cocircuit algebra of $V$} is the quotient $\cAlgebra{V}:=\polyx/\cIdeal{V}$.
\end{definition}

It is well-known that $\cAlgebra{V}$ has the same Hilbert series as $\pAlgebra{V}$ (see e.g.~\cite{Ber10,DM85,HR11}). We give a proof of this fact based on Theorem~\ref{thm:tutteHilbert} and on some elementary results in Stanley-Reisner theory. 

\begin{definition}
 Let $M$ be a matroid on a ground set $S$. The \defn{Stanley-Reisner ideal of $M$} is the ideal $\SRI{M}\subseteq \polyy$ generated by the monomials $\{\prod_{s\in T} y_s\}$, as $T\subseteq S$ ranges over the circuits of $M$.
The \defn{Stanley-Reisner ring of $M$} is the quotient ring $\kk[M]:=\polyy/\SRI{M}$. 
\end{definition} 

The following Lemma collects the preliminary results from Stanley-Reisner theory needed in the sequel. These are adaptations of more general statements to the particular context of representable matroids, relevant for our purposes.

\begin{lemma}
\label{lem:preliminarySR}
 Let $V=(v_s\colon s\in S)\subset \kk^d$ be a vector configuration which spans $\kk^d$ and $M=\matroid{V}$ be its rank-$d$ matroid.
 \begin{enumerate}[label=\thelemma(\alph*)]
 \item \label{it:lsop} The following $d$ linear forms are a linear system of parameters for $\kk[M]$:
 \[
 \theta_j:=\sum_{s\in S} (v_s)_jy_s,\quad j\in[d]
 \] 
where, as in Section~\ref{sec:pIdeals}, $(v_s)_j$ denotes the $j$-th coordinate of vector $v_s\in V$ \textup{(\cite[Lemma III.2.4]{Sta96})}.
 \item $\Hilb{\kk[M]/(\theta_1,\ldots,\theta_d)}=z^d\tutte{\dual{M}}(1,z^{-1})$, where $\dual{M}$ denotes the rank-$(n-d)$ matroid dual to $M$ \textup{(\cite[Equation (7.10) ff.]{Bjo92}, \cite[Theorem A3]{DP08})}.
 \item \label{it:spanningset} The quotient ring $\kk[M]/(\theta_1,\ldots,\theta_d)$ is spanned by the monomials $\{\prod_{s\in T}y_s$\}, where $T$ ranges over the independent sets of $M$ \textup{(\cite[Theorem III.2.5 ff.]{Sta96})}.
 \end{enumerate} 
\end{lemma} 

\begin{theorem}
\label{thm:isomorphism}
With the hypotheses of Lemma~\ref{lem:preliminarySR}, let $W=\left(w_s\colon s\in S\right)\subset\kk^{n-d}$ be such that $\matroid{W}=\dual{M}$, and $\theta_i:=\sum_{s\in S}(w_s)_iy_s$, for $ i\in[n-d]$, be the linear system of parameters for $\kk[\dual{M}]$ constructed from $W$ as in~\ref{it:lsop}. Then
\[
\cAlgebra{V}\cong\kk[\dual{M}]/(\theta_1,\ldots,\theta_{n-d}).
\] 
In particular, $\Hilb(\cAlgebra{V};z)=z^{n-d}\tutte{\matroid{V}}\left(1,{z}^{-1}\right)$, and $\cAlgebra{V}$ is spanned as a $\kk$-vector space by the products $
\{\prod_{s\in T} v_s(x)
\}$, where $T$ ranges over subsets $T\subseteq S$ with $\rank{\matroid{V}}(S\setminus T)=d$.
\end{theorem}

\begin{proof}
For notational convenience, let us first identify the common ground set $S$ of $M$ and $\dualM$ with $[n]$, and assume without loss of generality that the set $\{1,2,\ldots, d\}\subset[n]$ is a basis of $M$, so $\{d+1,\ldots,n\}$ is a basis of $\dualM$.
Let $g^{-1}\in\glnk$ be a transformation acting on $\polyyn$ as $g^{-1}\colon y_{d+i}\mapsto \theta_i$ for $1\leq i\leq n-d$. We choose $g^{-1}$ so that it has the following matrix form when expressed in the basis $\{y_1,\ldots,y_n\}$, :
\[
g^{-1}=\begin{pmatrix}
B_1 & &\0_{d\times(n-d)}\\
w_1 & \ldots & w_n\\
\kern.6em\vline & & \kern.6em\vline
\end{pmatrix}, \text{ so that }%
g=\begin{pmatrix}
v_1^\T \rule[.5ex]{3em}{0.4pt}\ & \0_{d\times(n-d)} \\
\vdots  & \\
v_n^\T \rule[.5ex]{3em}{0.4pt}\ & B_2 
\end{pmatrix} ,
\]
where $\0_{d\times(n-d)}$ denotes a $d\times(n-d)$ matrix of zeros, and $B_1\in\kk^{d\times d},B_2\in\kk^{(n-d)\times (n-d)}$ are suitable nonsingular matrices (which exist, by our assumption that $\{1,\ldots,d\}$ is a basis of $M$, and are uniquely determined). It follows that $g\cdot \kk[\dualM]/(\theta_1,\ldots,\theta_{n-d})\cong \cAlgebra{V}$, which establishes the isomorphism. The remaining statement follows by acting with $g$ on the spanning set in~\ref{it:spanningset}.
\end{proof}

\begin{proof}[Proof of Lemma~\ref{lem:powerkernel}]
Clearly, the ideal of relations $\ker \ringH{V}\subset\polyx$ of the squarefree algebra $\im(\ringH{V})$ contains the power ideal $\pIdeal{V}$. 
Indeed, the linear form $\ell_H$ associated to a hyperplane $H$ of $V$ maps to a linear form $\ringH{V}(\ell_H)\in\sfalgebra{V}$ with $s$-th coefficient equal to $\ell_H(v_s)$, which vanishes if and only if $v_s\in H$. It follows that the nonvanishing coefficients of $\ringH{V}(\ell_H)$ are indexed by elements in the complement of $H$ in $V$, which is a cocircuit $T$ of $V$. Since the only squarefree term of $\ringH{V}(\lPowerH)$ is a scalar multiple of $\prod_{s\in T}y_s$, we get $\lPowerH\in\ker \ringH{V}$. In particular, this implies the following inequality, understood coefficientwise:
\begin{equation}
\label{eq:hilbIneq}
\Hilb(\im(\ringH{V});z)\leq\Hilb(\pAlgebra{V};z).
\end{equation}

To prove the containment $\ker \ringH{V}\subseteq \pIdeal{V}$, we reproduce the linear-algebraic argument in~\cite[Chapter 4]{Ard03} and~\cite[Section 9]{PS04} to establish the equality of the dimensions of the graded components of $\im(\ringH{V})$ and $\cAlgebra{V}$. Then, by Theorems~\ref{thm:tutteHilbert} and~\ref{thm:isomorphism}, inequality~\eqref{eq:hilbIneq} holds with equality, so $\ker \ringH{V}= \pIdeal{V}$.

By Theorem~\ref{thm:isomorphism}, the $k$-th graded component $\cAlgebra{V}_k$ of $\cAlgebra{V}$ is spanned by the products
\[
\prod_{s\in T} v_s(x):=\prod_{s\in T} \left(\sum_{j\in [d]} (v_s)_j x_j\right),
\] 
where $T\subseteq S$ ranges over subsets with $|T|=k$ and $\rank{\matroid{V}}(S\setminus T)=d$. On the other hand, the $k$-th graded component $\im(\ringH{V})_k$ of $\im(\ringH{V})$ is spanned by the products 
\[
\prod_{j\in [d]}r_j^{q_j}:=\prod_{j\in [d]} \left(\sum_{s\in S}(v_s)_jy_s\right)^{q_j},
\]
where $q\in\nnZd$ ranges over exponent vectors with $\sum_{j\in[d]}q_j=k$. 

Given $T\subseteq S$ with $|T|=k$ and $\rank{\matroid{V}}(S\setminus T)=d$, and $q\in\nnZd$ with $\sum_{j\in[d]}q_j=k$, denote by $\lambda_{T,q}\in\kk$ the coefficient of $x^q$ in the expansion of $\prod_{s\in T}v_s(x)$ and by $\mu_{T,q}\in\kk$ the coefficient of $\prod_{s\in T}y_s$ in the expansion of $\prod_{j\in [d]}r_j^{q_j}$. Then $\lambda_{T,q}=\mu_{T,q}$. The claim follows since the dimension of $\cAlgebra{V}_k$ (resp. $\im(\ringH{V})_k$) is given by the rank of the matrix with rows labeled by $\{T\subseteq S\colon |T|=k,\ \rank{\matroid{V}}(S\setminus T)=d\}$, columns labeled by $\{q\in\nnZd\colon \sum_{j\in[d]}q_j=k\}$, and entries $\lambda_{T,q}$ (resp. $\mu_{T,q}$). 
\end{proof}

\begin{remark}
\label{rem:conjStanley}
A remarkable consequence in Stanley-Reisner theory\footnote{Which holds more generally for Stanley-Reisner rings of Cohen-Macaulay simplicial complexes modulo linear systems of parameters.} is that the (nonnegative integer) coefficients of the Hilbert series 
\[
\Hilb\left(\kk[\dualM]/(\theta_1,\ldots,\theta_{n-d});z\right)=h_0+h_1z+\ldots+h_{n-d}z^{n-d}
\]
are precisely the entries of the \defn{$h$-vector} $(h_0,\ldots,h_{n-d})$ of $\dualM$, a combinatorial invariant of the \defn{independence complex} of $\dualM$, which is defined as the simplicial complex of independent sets of the matroid $\dualM$. 

The study of numerical properties of $h$-vectors (such as log-concavity or unimodality) and of other combinatorial invariants of matroids is an active subject of research that has experienced major breakthroughs in recent years (e.g.~\cite{AHK17}). 

In this regard, Stanley conjectured in 1977 that $h$-vectors of matroids are \defn{pure $O$-sequences} (cf.~\cite[Conjecture III.3.6]{Sta96}). This means that given a matroid $h$-vector $(h_0,h_1,\ldots)$, there is a set of monomials $\cQ\subset\kk[x_1,\ldots,x_{h_1}]$ such that 
\begin{enumerate*}[label=(\roman*)]
\item if $m_1\in \cQ$ and $m_2\mid m_1$, then $m_2\in \cQ$ (that is, $\cQ$ is an \defn{order ideal of monomials}),
\item $\cQ$ contains exactly $h_k$ monomials of degree $k$, for $k\in\nnZ$, and
\item the maximal monomials of $\cQ$ with respect to divisibility have the same degree (that is, $\cQ$ is~\defn{pure}).
\end{enumerate*}

It is well-known (and not difficult to prove) that the integer vectors of a polymatroid can be regarded as the exponent vectors of a pure order ideal of monomials (see e.g.~\cite[Theorem 44.5]{Schr03}). Therefore, in light of Theorem~\ref{thm:tutteHilbert} and Lemma~\ref{lem:preliminarySR}, Corollary~\ref{cor:polymatroidmonomials} implies that Stanley's conjecture holds for the family of matroids dual to transversal matroids, known as  \defn{strict gammoids} or \defn{cotransversal matroids}. This fact had already been established by Oh in~\cite{Oh13}. In the author's opinion, it is rather surprising that Oh's proof also relies on the construction of a polymatroid associated to a set system, even though his methods are completely different.

Incidentally (and seemingly unbeknownst to them), the work of Postnikov and Shapiro also brought about a new proof of Stanley's conjecture for the family of matroids dual to graphic matroids, which had originally been settled by Merino using the language of the chip firing game on graphs~\cite{Mer01}.

\noindent Are there further instances of Stanley's conjecture that might yield to power ideals?
\end{remark}


\bibliographystyle{abbrv}
\bibliography{matroids}

\end{document}